\documentclass[11pt]{amsart}

\usepackage{amsmath,amssymb,amscd,amsthm,graphicx,enumerate}
\usepackage{hyperref}
\usepackage{xypic}
\usepackage{xcolor}
\usepackage{color}
\hypersetup{breaklinks=true}
\hypersetup{
    colorlinks=true,
    linkcolor=red,
    filecolor=magenta,      
    urlcolor=cyan,
    citecolor=green,
}

\numberwithin{equation}{section}
\theoremstyle{plain}

\newtheorem{theorem}{Theorem}[section]
\newtheorem{proposition}[theorem]{Proposition}
\newtheorem{lemma}[theorem]{Lemma}
\newtheorem{corollary}[theorem]{Corollary}

\newtheorem{question}[theorem]{Question}

\theoremstyle{remark}

\theoremstyle{definition}

\renewcommand{\cH}{{\mathcal H}}

\begin{document}

\title[A nonexistence result for rotating flows in $\mathbb{R}^{4}$]{A nonexistence result for rotating mean curvature flows in $\mathbb{R}^{4}$}

\author{Wenkui Du, Robert Haslhofer}

\begin{abstract}
Some worrisome potential singularity models for the mean curvature flow are rotating ancient flows, i.e. ancient flows whose tangent flow at $-\infty$ is a cylinder $\mathbb{R}^k\times S^{n-k}$ and that are rotating within the $\mathbb{R}^k$-factor. We note that while the $\mathbb{R}^k$-factor, i.e. the axis of the cylinder, is unique by the fundamental work of Colding-Minicozzi, the uniqueness of tangent flows by itself does not provide any information about rotations within the $\mathbb{R}^k$-factor. In the present paper, we rule out rotating ancient flows among all ancient noncollapsed flows in $\mathbb{R}^4$.\end{abstract}

\maketitle

\bigskip

\section{Introduction}

To capture the formation of singularities in geometric flows one always magnifies the original flow by rescaling by a sequence of factors going to infinity and passes to a blowup limit. Any such blowup limit is an ancient solution, i.e. a solution that is defined for all sufficiently negative times. In particular, analyzing blowup limits of mean curvature flow near any cylindrical singularity,\footnote{In particular, all generic singularities are expected to be cylindrical \cite{CM_generic,CCMS}.} one is indispensably led to study ancient asymptotically cylindrical flows, i.e. mean curvature flows $M_t\subset\mathbb{R}^{n+1}$ that are defined for all times $t\ll 0$, and whose tangent flow at $-\infty$ is a cylinder, namely
\begin{equation}\label{tangent_minus_infty}
\lim_{\lambda \rightarrow 0} \lambda M_{\lambda^{-2}t}=\mathbb{R}^{k}\times S^{n-k}(\sqrt{2(n-k)|t|})
\end{equation}
for some $1\leq k\leq n-1$. While in the neck-case, i.e. for $k=1$, a complete classification of such flows has been obtained recently in \cite{CHH,CHHW}, a classification in the general case $k\geq 2$ seems currently out of reach.

By the fundamental work of Colding-Minicozzi \cite{CM_uniqueness}, cylindrical tangent flows are unique, i.e. the limit in \eqref{tangent_minus_infty} does not depend on any choice of subsequence $\lambda_i\to 0$. In particular, the axis of the cylinder, i.e. the $\mathbb{R}^k$-factor is unique. Another foundational result is the Brendle-Choi neck-improvement theorem \cite{BC1,BC2}, together with its recent generalizations for $k\geq 2$ from \cite{Zhu1,Zhu2,DZ}, which in many cases allows one to promote the $\mathrm{SO}_{n-k+1}$ symmetry of the tangent flow at $-\infty$ to $\mathrm{SO}_{n-k+1}$ symmetry of the ancient asymptotically cylindrical flow. On the other hand, these results do not provide any information about what is happening within the $\mathbb{R}^k$-factor. Specifically, an important open question is the following.

\begin{question}[rotating ancient flows]\label{question_rotating}
Are there any ancient asymptotically cylindrical flows that are rotating within the $\mathbb{R}^k$-factor?
\end{question}

In particular, this includes the potential scenario of compact solutions with approximately ellipsoidal shape that slowly rotate (and shrink). While the question seems folklore, we are not aware of any formal definition of ``rotating within the $\mathbb{R}^k$-factor". For our present purpose, we instead take the more pragmatic interpretation that to show that a certain ancient asymptotically cylindrical flow is not ``rotating within the $\mathbb{R}^k$-factor" one has to come up with a convincing reason why it is clearly not, specifically:
\begin{itemize}
\item either establish $\mathrm{SO}_{k}$ symmetry,
\item or identify some geometrically distinguished subspaces $V^\ell\subset\mathbb{R}^k$.
\end{itemize}

An example of the former would be the $\mathrm{SO}_{k}\times \mathrm{SO}_{n-k+1}$ symmetric ancient ovals from \cite{HaslhoferHershkovits_ancient}, which clearly cannot (visibly) rotate within the $\mathbb{R}^k$-factor thanks to their $\mathrm{SO}_{k}$ symmetry. An example of the latter, would be the $k-1$ parameter family of $\mathbb{Z}_2^k\times \mathrm{SO}_{n-k+1}$ symmetric ancient ovals from our recent paper \cite{DH_ovals}, where we have $k$ perpendicular lines $L_1,\ldots,L_k\subset\mathbb{R}^k$ of reflection symmetry, which again clearly prevent rotations within the $\mathbb{R}^k$-factor.

In the present paper, we will rule out rotating ancient flows among all ancient noncollapsed flows $M_t$ in $\mathbb{R}^4$. Recall that noncollapsing, a central property in singularity analysis, means that the flow is moving inwards and there is some constant $\alpha>0$ such that every point $p\in M_t$ admits interior and exterior balls of radius at least $\alpha/H(p)$, c.f. \cite{ShengWang,Andrews_noncollapsing,HaslhoferKleiner_meanconvex}. It is known that all blowup limits of mean-convex flows are ancient noncollapsed flows, see \cite{White_size,White_nature,White_subsequent,HH_subsequent}.  More generally, by Ilmanen's mean-convex neighborhood conjecture \cite{Ilmanen_problems}, which has been proved recently in the case of neck-singularities in \cite{CHH,CHHW}, it is expected that in fact every ancient asymptotically cylindrical flow is noncollapsed.

By general theory \cite{HaslhoferKleiner_meanconvex}, for ancient noncollapsed flows in $\mathbb{R}^4$ the tangent flow at $-\infty$ is always either a round shrinking sphere, a round shrinking neck, a round shrinking bubble-sheet or a static plane. Since the static plane is trivial and since rotations can only happen for $k\geq 2$, we can thus assume from now on that the tangent flow at $-\infty$ is a bubble-sheet, namely
\begin{equation}\label{bubble-sheet_tangent_intro}
\lim_{\lambda \rightarrow 0} \lambda M_{\lambda^{-2}t}=\mathbb{R}^{2}\times S^{1}(\sqrt{2|t|}).
\end{equation}
The key to rule out rotating ancient noncollapsed flows in $\mathbb{R}^4$ is to come up with a theorem that captures the deviation from the round bubble-sheet. Some progress in this direction has been made in our recent paper \cite{DH_hearing_shape}. To describe this, recall that \eqref{bubble-sheet_tangent_intro} equivalently means that the renormalized flow $\bar M_\tau = e^{\frac{\tau}{2}}  M_{-e^{-\tau}}$
for $\tau\to -\infty$ converges to $\Gamma=\mathbb{R}^2\times S^{1}(\sqrt{2})$.
Writing $\bar{M}_\tau$ as a graph of a function $u(\cdot,\tau)$ over $\Gamma\cap B_{\rho(\tau)}$, where $\rho(\tau)\to \infty$ as $\tau\to -\infty$, namely
\begin{equation}\label{graph_over_cylinder}
\left\{ q+ u(q,\tau)\nu(q) \, : \, q\in \Gamma\cap B_{\rho(\tau)} \right\} \subset \bar M_\tau\, ,
\end{equation}
where $\nu$ denotes the outwards unit normal of $\Gamma$, we proved:

\begin{theorem}[{bubble-sheet quantization \cite{DH_hearing_shape}}]\label{spectral theorem_restated}
For any ancient noncollapsed mean curvature flow in $\mathbb{R}^{4}$ whose tangent flow at $-\infty$ is given by \eqref{bubble-sheet_tangent_intro}, the bubble-sheet function $u$ satisfies
 \begin{equation}\label{main_thm_ancient}
\lim_{\tau\to -\infty} \Big\|\,
|\tau| u({\bf y},\vartheta,\tau)- {\bf y}^\top Q(\tau){\bf y} +2\mathrm{tr}(Q(\tau))\, \Big\|_{C^{k}(B_R)} = 0
\end{equation}
for all $R<\infty$ and all integers $k$, where $Q(\tau)$ is a symmetric $2\times 2$-matrix, whose eigenvalues are quantized to be either 0 or $-1/\sqrt{8}$, and which possibly depends on the time $\tau$. More precisely, for $\mathrm{rk}(Q(\tau))\neq 1$ the matrix $Q(\tau)$ is actually independent of $\tau$, while in the case $\mathrm{rk}(Q(\tau))=1$ we have
\begin{equation}
Q(\tau)=R(\tau)^\top \begin{pmatrix}
0 & 0\\
0 & -1/\sqrt{8}
\end{pmatrix}R(\tau)
\end{equation}
for some rotation matrix $R(\tau)\in \mathrm{SO}_2$ with $|\dot{R}(\tau)|=o(|\tau|^{-1})$.
\end{theorem}
Intuitively, the theorem strongly suggests that in the rank 1 case there should be a distinguished line $L\subset \mathbb{R}^2$, which geometrically speaking captures the direction in which there is some inwards quadratic bending. However, in our prior paper the error terms were unfortunately so large that we only obtained the estimate $|\dot{R}(\tau)|=o(|\tau|^{-1})$, which is barely not integrable, and thus did not allow us to rule out rotations within the $\mathbb{R}^2$-factor.

In the present paper, we crucially improve this theorem and show that the fine-bubble sheet matrix can in fact be taken independent of time:

\begin{theorem}[{bubble-sheet quantization - improved version}]\label{spectral theorem_improved}
For any ancient noncollapsed mean curvature flow in $\mathbb{R}^{4}$ whose tangent flow at $-\infty$ is given by \eqref{bubble-sheet_tangent_intro}, the bubble-sheet function $u$ satisfies
 \begin{equation}\label{main_thm_ancient}
\lim_{\tau\to -\infty} \Big\|\,
|\tau| u({\bf y},\vartheta,\tau)- {\bf y}^\top Q{\bf y} +2\mathrm{tr}(Q)\, \Big\|_{C^{k}(B_R)} = 0
\end{equation}
for all $R<\infty$ and all integers $k$, where $Q$ is a constant symmetric $2\times 2$-matrix whose eigenvalues are quantized to be either 0 or $-1/\sqrt{8}$.
\end{theorem}

A very important inspiration for our present result is the work by Filippas-Liu \cite{FilippasLiu}, who proved a related quantization result (forwards in time, of course with the opposite sign) for singularities of multidimensional semilinear heat equations. However, controlling the error terms in the setting of mean curvature flow turned out to be quite a bit more delicate. We also note that while the Lojasiewicz inequality from Colding-Minicozzi \cite{CM_uniqueness} does not provide any information about rotations within the $\mathbb{R}^2$-factor, it is nevertheless a key ingredient for our analysis. Specifically, the uniqueness of cylindrical tangent flows together with an improved graphical radius estimate are crucial to get our fine bubble-sheet analysis started.

As a consequence of Theorem \ref{spectral theorem_improved}, taking also into account work in the rank 0 case from \cite{CHH_wing,DH_hearing_shape} and work in the rank 2 case from \cite{CDDHS}, we obtain:

\begin{corollary}[nonexistence of rotating ancient flows]\label{cor_nonexistence}
Every ancient noncollapsed flow in $\mathbb{R}^4$, whose tangent flow at $-\infty$ is given by \eqref{bubble-sheet_tangent_intro}, 
\begin{itemize}
\item either is $\mathrm{SO}_{2}$ symmetric within the $\mathbb{R}^2$-factor,
\item or possesses one or two geometrically distinguished line $L\subset\mathbb{R}^2$ or lines $L_1,L_2\subset\mathbb{R}^2$, respectively, of reflection symmetry,
\item or possesses a geometrically distinguished line $L\subset\mathbb{R}^2$ of inwards quadratic bending.
\end{itemize}
In particular, there are no rotating ancient noncollapsed flows in $\mathbb{R}^4$.
\end{corollary}

In particular, the corollary removes a stumbling block in the classification program, introduced in \cite{CHH_wing,DH_hearing_shape}, for ancient noncollapsed flows in $\mathbb{R}^4$.

\bigskip

To outline our proof, we recall from \cite{CIM,DH_hearing_shape} that the analysis of the renormalized flow over the bubble-sheet $\Gamma=\mathbb{R}^2\times S^{1}(\sqrt{2})$ is governed by an Ornstein-Uhlenbeck type operator $\mathcal{L}$, which is a self-adjoint operator on the Hilbert space of Gaussian $L^2$ functions on $\Gamma$. Similarly as in \cite{DH_hearing_shape}, we consider the spectral coefficients $\alpha_1,\alpha_2,\alpha_3$ obtained by taking the inner product of a certain truncated version of the bubble-sheet function $u$ with certain neutral eigenfunctions of $\mathcal{L}$. While the asymptotic behavior, as in the statement of the bubble-sheet quantization theorem, can be guessed quite easily by formally deriving ODEs for these spectral coefficients, the crux of the matter to actually prove the theorem is to derive good enough error estimates. To this end, we first prove a quantitative version of the Merle-Zaag lemma. We then derive an improved error estimate for our spectral ODEs, which crucially improves our prior nonintegrable $o(|\tau|^{-1})$ estimate into an integrable $O(|\tau|^{-1-\delta})$ estimate. Finally, we suitably integrate our spectral ODEs with the improved error estimate to conclude the proof.
\bigskip

\noindent\textbf{Acknowledgments.}
This research was supported by the NSERC Discovery Grant and the Sloan Research Fellowship of the second author.

\bigskip

\section{The proof}
We recall from \cite{CIM,DH_hearing_shape} that the analysis of the renormalized flow $\bar{M}_\tau=e^{\tau/2}M_{-e^{-\tau}}$ over $\Gamma=\mathbb{R}^2\times S^{1}(\sqrt{2})$ is governed by the Ornstein-Uhlenbeck type operator
\begin{equation}
\mathcal L=\frac{\partial^2}{\partial y_1^2}-\frac{y_1}{2}\frac{\partial}{\partial y_{1}}
+\frac{\partial^2}{\partial y_2^2}-\frac{y_2}{2}\frac{\partial}{\partial y_{2}}
+\frac{1}{2}\frac{\partial^2}{\partial \vartheta^2}+1.
\end{equation}
We denote by $\mathcal{H}$ the Hilbert space of Gaussian $L^2$ functions on $\Gamma$, equipped with the inner product
\begin{equation}\label{def_norm}
\langle f, g\rangle_{\mathcal{H}}= \frac{1}{(4\pi)^{3/2}} \int_{\Gamma}  f(q)g(q) e^{-\frac{|q|^2}{4}} \, dq\, .
\end{equation}
Analyzing the spectrum of $\mathcal L$, we can decompose our Hilbert space as
\begin{equation}
\mathcal H = \mathcal{H}_+\oplus \mathcal{H}_0\oplus \mathcal{H}_-,
\end{equation}
where the unstable space is given by
\begin{align}
\mathcal{H}_+ =\text{span}\{1, y_1, y_{2}, \cos\vartheta, \sin\vartheta\},\label{basis_hplus}
\end{align}
and neutral space is given by
\begin{align}
\mathcal{H}_0 =\text{span}\left\{y^2_{1}-2, y^2_{2}-2, y_{1}y_{2}, y_{1}\cos\vartheta, y_{1}\sin\vartheta,  y_{2}\cos\vartheta, y_{2}\sin\vartheta \right\}.
\label{basis_hneutral}
\end{align}
By \cite[Proposition 2.5]{DH_hearing_shape}, which has been proved using the Lojasiewicz inequality from \cite{CM_uniqueness}, there exist $\gamma\in(0,1)$ and $\tau_\ast>-\infty$ such that $\rho(\tau)=|\tau|^{\gamma}$ is an admissible graphical radius for $\tau\leq \tau_\ast$, so in particular $\bar{M}_\tau$ can be written as a graph of a function $u(\cdot,\tau)$ over $\Gamma\cap B_{2\rho(\tau)}$ with the estimate
\begin{equation}\label{graph_rad_ass}
\| u(\cdot,\tau) \|_{C^4(\Gamma\cap B_{2\rho(\tau)}(0))}\leq \rho(\tau)^{-2}.
\end{equation}
As usual, fixing a smooth cutoff function $\chi$ satisfying $\chi(s)=1$ for $s\leq 1$ and $\chi(s)=0$ for $s\geq 2$, we work with the truncated graph function
\begin{equation}
\hat{u}(\cdot,  \tau)=u(\cdot,  \tau)\chi\left(\frac{|\cdot|}{|\tau|^\gamma}\right),
\end{equation}
and consider
\begin{align}
U_0(\tau) := \|\mathcal{P}_0 \hat{u}(\cdot,\tau)\|_{\mathcal{H}}^2,\qquad U_{\pm}(\tau) := \|\mathcal{P}_{\pm} \hat{u}(\cdot,\tau)\|_{\mathcal{H}}^2,
\end{align}
where $\mathcal{P}_{0}$ and $\mathcal{P}_{\pm}$ are the orthogonal projections to $\mathcal{H}_0$ and $\mathcal{H}_{\pm}$ respectively. Then, by \cite[Equation (2.24)]{DH_hearing_shape} we have the differential inequalities
\begin{align}\label{modes_inequality}
&\dot{U}_+ \geq U_+ - \Lambda |\tau|^{-\gamma} \, (U_+ + U_0 + U_-),\nonumber \\
&\big | \dot{U}_0 \big | \leq \Lambda |\tau|^{-\gamma} \, (U_+ + U_0 + U_-), \\\
&\dot{U}_- \leq -U_- + \Lambda |\tau|^{-\gamma} \, (U_+ + U_0 + U_-) \nonumber
\end{align}
for $\tau\leq\tau_\ast$, where $\Lambda <\infty$ is a numerical constant. By the Merle-Zaag alternative from \cite[Proposition 2.2]{DH_hearing_shape} for $\tau\to -\infty$ either the neutral mode is dominant or the unstable mode is dominant. If the unstable mode is dominant, then, as explained in \cite[Page 17]{DH_hearing_shape}, the conclusion of the bubble-sheet quantization theorem holds with $Q=0$, so we can assume from now on that the neutral mode is dominant, i.e. that for $\tau\to -\infty$ we have
\begin{equation}\label{eq_neutral_dom}
U_{-}+U_{+}=o(U_0).
\end{equation}
For our present purpose, it is important to improve this to a more quantitative estimate:

\begin{lemma}[{quantitative Merle-Zaag type estimate}]\label{quant_MZ}
There exist $\tau_0\leq \tau_\ast$ and $C_0<\infty$, such that for all $\tau\leq \tau_0$ we have
\begin{equation}\label{zero_dom_mz}
U_{-}(\tau)+U_+(\tau) \leq \frac{C_0}{|\tau|^\gamma}U_{0}(\tau).
\end{equation}
\end{lemma}

\begin{proof} Our argument is closely related to the proof of the classical Merle-Zaag lemma \cite{MZ}, but with some modifications to obtain a better decay. Specifically, here we work with the time-dependent function
\begin{equation}
\lambda(\tau):=\frac{\Lambda}{|\tau|^{\gamma}},
\end{equation}
where $\Lambda$ and $\gamma$ are the constants from above. Possibly after decreasing the constant $\tau_\ast$, we may assume that for $\tau\leq \tau_\ast$ we have $\lambda\leq \tfrac{1}{10}$ and also
\begin{equation}\label{small_der}
\dot{\lambda}\leq \tfrac{1}{10}\lambda.
\end{equation}

Now, if at some time $\bar{\tau}\leq \tau_\ast$ the quantity $f:=U_{-}-2\lambda (U_{+}+U_{0})$ was positive, then using \eqref{modes_inequality} and $\dot{\lambda}\geq 0$ at this time we would get
\begin{equation}
    \dot{f}\leq -U_{-}+\lambda(1+4\lambda)\left(1+\frac{1}{2\lambda}\right)U_{-}-2\dot{\lambda}(U_{+}+U_{0})\leq 0,
\end{equation}
which would imply that $ f(\tau)\geq f(\bar{\tau})>0$
for all $\tau\leq \bar{\tau}$, contradicting the fact that $\lim_{\tau\rightarrow -\infty}f(\tau)=0$. This shows that for all $\tau\leq \tau_\ast$ we have
\begin{equation}\label{-leq0+}
    U_{-}\leq 2\lambda(U_{0}+U_{+}).
\end{equation}

Next, we consider the quantity
\begin{equation}
g:=8\lambda U_0-U_{+}.
\end{equation}

If we had $g(\tau)\leq 0$ for all $\tau\leq \tau_\ast$, then \eqref{modes_inequality} and \eqref{-leq0+} would imply
\begin{equation}
\dot{U}_{+}\geq \frac14 U_{+},\qquad |\dot{U}_0|\leq 2\lambda(U_0+U_+)\leq \left(\frac14 +2\lambda\right)U_{+}.
\end{equation}
Integrating the first differential inequality from $-\infty$ to $\tau$ would give
\begin{equation}
U_{+}(\tau)\geq\frac14 \int_{-\infty}^\tau U_{+}(\sigma)\, d\sigma .
\end{equation}
Using this and the monotonicity of $\lambda$, integrating the second differential inequality from $-\infty$ to $\tau$ would yield
\begin{equation}
U_{0}(\tau) \leq (1+8\lambda(\tau))U_{+}(\tau),
\end{equation}
contradicting \eqref{eq_neutral_dom}. Thus, there exists some $\tau_0\leq \tau_\ast$ such that  $g(\tau_0)>0$.

Finally, if the inequality $g>0$ failed  at some time less than $\tau_0$, then at the largest time $\bar\tau<\tau_0$ where it failed we would have $g(\bar\tau)=0$ and by \eqref{modes_inequality}, \eqref{small_der} and \eqref{-leq0+} we would get
\begin{align}
\dot{g} &\leq \lambda(8\lambda+1)(U_{+}+U_0+U_-)-U_{+} { +}8\dot{\lambda}U_0\nonumber\\
&\leq \lambda(8\lambda+1)(8\lambda+1+2\lambda (1+8\lambda ))U_0 - 8(\lambda{ -}\dot{\lambda} )U_0\nonumber\\
&\leq -(\lambda {-} 8\dot{\lambda} )U_0<0,
\end{align}
contradicting the definition of $\bar\tau$. This shows that for all $\tau\leq \tau_0$ we have
\begin{equation}
U_{+} < 8\lambda U_0.
\end{equation}
Remembering \eqref{-leq0+}, this concludes the proof of the lemma.
\end{proof}

Now, as in \cite[Section 3.1]{DH_hearing_shape}, we consider the spectral coefficients
\begin{equation}
    \alpha_{j}(\tau):=  \frac{\langle   \hat{u}(\cdot,\tau), \psi_{j} \rangle_{\cH}}{|| \psi_{j} ||_{\cH}^{2}},
\end{equation}
with respect to the neutral eigenfunctions
\begin{equation}
\psi_{1}=y_1^2-2,\quad \psi_{2}=y_2^2-2, \quad \psi_{3}=2y_1y_2.
\end{equation}

It has been shown in \cite[Proposition 3.1]{DH_hearing_shape} that the spectral coefficients $\vec{\alpha}=(\alpha_1,\alpha_2,\alpha_3)$ satisfy the ODE system
\begin{equation}\label{ODEs}
 \begin{cases}
   \dot{\alpha}_{1}=-\sqrt{8}(\alpha^2_{1}+\alpha_{3}^2)+E_{1}\\
   \dot{\alpha}_{2}=-\sqrt{8}(\alpha^2_{2}+\alpha_{3}^2)+E_{2}\\
    \dot{\alpha}_{3}=-\sqrt{8}(\alpha_{1}+\alpha_2)\alpha_{3}+E_{3},\\
    \end{cases}
\end{equation}
where the error terms satisfy
\begin{equation}
|E_{j}(\tau)|=o(|\vec{\alpha}(\tau)|^2+|\tau|^{-100}).
\end{equation}
We also recall that by \cite[Claim 3.5]{DH_hearing_shape} for $\tau$ sufficiently negative one has
\begin{equation}
|\vec{\alpha}(\tau)|^2\geq \frac{1}{10|\tau|^2}.
\end{equation}
Here, we improve the error estimate as follows:

\begin{proposition}[improved error estimate]\label{improved_error}
There exist $\delta>0$, $\tau_0>-\infty$ and $C<\infty$, such that for all $\tau\leq\tau_0$ the error terms in \eqref{ODEs} satisfy
\begin{equation}
|E_{j}(\tau)|\leq \frac{C}{|\tau|^{2+\delta}}.
\end{equation}
\end{proposition}

\begin{proof}
Let us consider the remainder
\begin{equation}
w:=\hat{u}-\sum_{j=1}^3 \alpha_j \psi_j.
\end{equation}
In light of \cite[Proof of Proposition 3.1]{DH_hearing_shape} it is enough to show that
\begin{equation}\label{est_to_show}
\left| 2\sum_{i=1}^3 \alpha_i(\tau) \langle \psi_i\psi_k , w(\tau)\rangle_\cH + \langle w(\tau)^2,\psi_k \rangle_\cH\right| \leq  \frac{C}{|\tau|^{2+\delta}}.
\end{equation}
To do so, we start by applying Lemma \ref{quant_MZ} (quantitative Merle-Zaag type estimate), which yields
\begin{equation}\label{est_w}
\| w(\tau)\|_\cH^2 \leq \frac{C}{|\tau|^\gamma}\big(|\vec{\alpha}(\tau)|^2+R(\tau)\big),
\end{equation}
where 
\begin{equation}
R=\langle \hat u, y_1\cos\vartheta\rangle_{\cH}^2+\langle \hat u, y_1\sin\vartheta\rangle_{\cH}^2
+\langle \hat u, y_2\cos\vartheta\rangle_{\cH}^2+\langle \hat u, y_2\sin\vartheta\rangle_{\cH}^2.
\end{equation}
Moreover, by the a priori estimates from \cite[Proposition 3.3 and Claim 3.5]{DH_hearing_shape} we have
\begin{equation}\label{upperlower_alpha}
\frac{1}{C|\tau|}\leq |\vec{\alpha}(\tau)|\leq \frac{C}{|\tau|},
\end{equation}
and thanks to the almost circular symmetry from \cite[Proposition 2.7]{DH_hearing_shape} we get
\begin{equation}\label{est_R}
R(\tau)\leq \frac{C}{|\tau|^{100}}.
\end{equation}
Combining the above, and choosing $\delta=\tfrac12 \gamma$, we obtain
\begin{equation}\label{est_w2}
\| w(\tau)\|_\cH \leq \frac{C}{|\tau|^{1+\delta}},
\end{equation}
hence
\begin{equation}
\left| 2\sum_{i=1}^3 \alpha_i(\tau) \big\langle \psi_i\psi_k , w(\tau)\big\rangle_\cH \right| \leq  \frac{C}{|\tau|^{2+\delta}}.
\end{equation}
On the other hand, to control the quadratic term in $w$ in \eqref{est_to_show} we will prove a gradient estimate. To this end, recall from \cite[Proposition 2.8]{DH_hearing_shape} that 
\begin{equation}
(\partial_\tau-\mathcal{L})\hat{u}=-\frac{1}{\sqrt{8}}\hat{u}^2+E,
\end{equation}
where taking also into account the almost circular symmetry from \cite[Proposition 2.7]{DH_hearing_shape} the error term can be estimated by
\begin{equation}\label{hat_e_error}
\| E(\tau)\|_{\mathcal{H}} \leq \frac{C}{|\tau|^\gamma}\| \hat{u}\|^2_{\mathcal{H}}  + \frac{C}{|\tau|^{10}}\leq  \frac{C}{|\tau|^{2+\gamma}},
\end{equation}
and where in the last step we used \eqref{upperlower_alpha} and \eqref{est_w2}.
Considering the projection $\mathcal{P}^\perp$ to the orthogonal complement of $\mathrm{span}\{ \psi_1,\psi_2,\psi_3\}$ this yields
\begin{equation}\label{eq_w}
(\partial_\tau - \mathcal{L}) w = g,
\end{equation}
where
\begin{equation}
g=\mathcal{P}^\perp \left(-\frac{1}{\sqrt{8}}\hat{u}^2+E\right).
\end{equation}
Since projections do not increase the norm, using \eqref{hat_e_error} we can estimate
\begin{equation}
\| g(\tau)\|_{\mathcal{H}}\leq \frac{1}{\sqrt{8}}\| \hat{u}(\tau)^2\|_{\mathcal{H}}+\frac{C}{|\tau|^{2+\gamma}}.
\end{equation}
Moreover, using \eqref{graph_rad_ass}, and remembering also \eqref{upperlower_alpha} and \eqref{est_w2}, we see that
\begin{equation}
\| \hat{u}(\tau)^2\|_{\mathcal{H}}\leq \frac{1}{|\tau|^{2\gamma}}\| \hat{u}(\tau)\|_{\mathcal{H}}\leq \frac{C}{|\tau|^{1+2\gamma}}.
\end{equation}
Combining the above, we have thus shown that
\begin{equation}\label{est_g}
\| g(\tau)\|_{\mathcal{H}}\leq  \frac{C}{|\tau|^{1+2\gamma}}.
\end{equation}
Now, given any $\bar\tau\leq\tau_0-1$ using \eqref{eq_w} and integration by parts we compute
\begin{align}
\frac{d}{d\tau}\int {e^{\bar\tau-\tau}} w^2e^{-q^2/4} &= \int  {e^{\bar\tau-\tau}}(2wg-2|\nabla w|^2 -w^2)\, e^{-q^2/4}\nonumber\\
& \leq \int  {e^{\bar\tau-\tau}}(g^2-2|\nabla w|^2)\, e^{-q^2/4},
\end{align}
and
\begin{align}
 \frac{d}{d\tau}\int (\tau-\bar\tau)|\nabla w|^2 e^{-q^2/4}
 &=\int \left(|\nabla w|^2  -2(\tau-\bar\tau) (\mathcal{L}  w)(\mathcal{L}w+g)\right)\, e^{-q^2/4}\nonumber\\
 &\leq\int \left(|\nabla w|^2  +\tfrac{1}{2}(\tau-\bar\tau) g^2)\right)\, e^{-q^2/4}\, .
\end{align}
This yields
\begin{align}
 \frac{d}{d\tau}\int \left((\tau-\bar\tau)|\nabla w|^2+\tfrac{e^{\bar\tau-\tau}}{2} w^2\right)\, e^{-q^2/4}\leq \int g^2\, e^{-q^2/4}\, .
 \end{align}
Thus, together with \eqref{est_w2} and \eqref{est_g} for all $\tau\leq\tau_0$ we get
\begin{equation}
\|\nabla w(\tau)\|_{\mathcal{H}} \leq \frac{C}{|\tau|^{1+\delta}}.
\end{equation}
Hence, applying the weighted Poincare inequality \cite[page 109]{Ecker_logsob} we conclude that
\begin{equation}\label{est_to_show_done}
\big| \langle w(\tau)^2,\psi_k \rangle_\cH\big| \leq  \frac{C}{|\tau|^{2+\delta}}.
\end{equation}
This finishes the proof of the proposition.
\end{proof}

We can now conclude the proof of our main theorem:

\begin{proof}[{Proof of Theorem \ref{spectral theorem_improved}}]
In the setting from above, we consider
\begin{equation}
a:=-\sqrt{2}(\alpha_1+\alpha_2),\qquad b:= 8(\alpha_1\alpha_2-\alpha_3^2)\, \qquad c=\alpha_3.
\end{equation}
Using our spectral ODEs \eqref{ODEs}, we infer that
\begin{equation}\label{Riccati ODEs_ab}
    \begin{cases}
      \dot{a}=2a^2-b+E_a,\\
       \dot{b}=2ab+E_b,\\
       \dot{c}=2ac+E_c,\\
    \end{cases}
\end{equation}
where thanks to Proposition \ref{improved_error} (improved error estimate) we have
\begin{equation}
|E_a(\tau)|\leq \frac{C}{|\tau|^{2+\delta}}, \quad |E_b(\tau)|\leq \frac{C}{|\tau|^{3+\delta}},\quad |E_c(\tau)|\leq \frac{C}{|\tau|^{2+\delta}}.
\end{equation}
In fact, it is useful to make yet another substitution. Specifically, we set
\begin{equation}
x:=-\tau a\, , \quad y:=\tau^2 b\ ,\quad z=-\tau c\, , \qquad \sigma:=-\log(-\tau)\, .
\end{equation}
In these new variables our problem takes the form
\begin{equation}\label{Riccati ODEs_xy}
    \begin{cases}
      x'=2x^2-x-y+E_x,\\
      y'=2xy-2y+E_y,\\
      z'=2xz-z+E_z,
    \end{cases}
    \end{equation}
where
\begin{equation}
|E_x(\sigma)|+|E_y(\sigma)|+|E_z(\sigma)|\leq Ce^{\delta\sigma}.
\end{equation}
To proceed, we can assume that we are in the rank $1$ case, since otherwise there is nothing to prove. Then, by our prior weak version of the bubble-sheet quantization theorem \cite[Theorem 3.6]{DH_hearing_shape} we already know the asymptotics of the trace and the determinant, specifically we already know that
\begin{equation}
(\alpha_1+\alpha_2)(\tau)= - \frac{1}{\sqrt{8}|\tau|}+o(|\tau|^{-1}),\quad (\alpha_1\alpha_2-\alpha_3^2)(\tau)= o(|\tau|^{-2}).
\end{equation}
Transforming variables this becomes
\begin{equation}\label{apriori_xy}
\lim_{\sigma\to -\infty} x(\sigma) =  \frac{1}{2}\, , \qquad \lim_{\sigma\to -\infty} y(\sigma)=0 \, .
\end{equation}
In particular, setting
\begin{equation}
p(\sigma):=1-2x(\sigma),
\end{equation}
we have
\begin{equation}
\sup_{\sigma\leq \sigma_0}|p(\sigma)|\leq \frac{\delta}{3},
\end{equation}
possibly after decreasing $\sigma_0=-\log(-\tau_0)$. Thus, for any $\sigma\leq\sigma_0$ we obtain
\begin{equation}
\left|  e^{-\int_{\sigma}^{\sigma_0} p(\tilde{\sigma}) d\tilde{\sigma}} E_z(\sigma) \right| \leq C e^{\frac{2}{3}\delta \sigma},
\end{equation}
where $C=C(\sigma_0)<\infty$, and consequently
\begin{equation}
\left| \frac{d}{d\sigma} \left(  e^{-\int_{\sigma}^{\sigma_0} p(\tilde{\sigma}) d\tilde{\sigma}}z(\sigma) \right) \right| \leq C e^{\frac{2}{3}\delta \sigma}.
\end{equation}
Finally, by an orthogonal transformation of the $y_1y_2$-coordinates, we can arrange that $\alpha_3(\tau_0)=0$. Therefore, for any $\sigma\leq\sigma_0$  we conclude that
\begin{equation}
|z(\sigma)| \leq Ce^{\frac{1}{3}\delta \sigma}.
\end{equation}
Translating back to the original variables, this shows that for $\tau\leq \tau_0$ we have
\begin{equation}
|\alpha_3(\tau)|\leq C\frac{1}{|\tau|^{1+\frac13 \delta}}.
\end{equation}
Hence, the conclusion of the bubble-sheet quantization theorem holds with a constant diagonal matrix.
\end{proof}

Finally, let us prove the corollary.

\begin{proof}[{Proof of Corollary \ref{cor_nonexistence}}]
If $\mathrm{rk}(Q)=0$, then by \cite{CHH_wing,DH_hearing_shape} the flow is, up to scaling and rigid motion, either the round shrinking bubble-sheet, which is in particular $\mathrm{SO}_2$ symmetric in the $\mathbb{R}^2$-factor,  or $\mathbb{R}\times$2d-bowl, which is in particular reflection symmetric.

If $\mathrm{rk}(Q)=1$, then by Theorem \ref{spectral theorem_improved}, there is a distinguished line $L\subset\mathbb{R}^2$, namely the range of $Q$, of inwards quadratic bending.

Finally, if $\mathrm{rk}(Q)=2$, then by \cite{CDDHS} the flow is, up to scaling and rigid motion, either the $\mathrm{SO}_{2}\times \mathrm{SO}_{2}$ symmetric ancient oval from \cite{HaslhoferHershkovits_ancient}, or belongs to $1$ parameter family of $\mathbb{Z}_2^2\times \mathrm{SO}_{2}$ symmetric ancient ovals from \cite{DH_ovals}, so the conclusion holds in the rank 2 case as well.
\end{proof}

\bigskip

\bibliography{no_rotation}

\bibliographystyle{alpha}

\vspace{10mm}

{\sc Wenkui du, Department of Mathematics, University of Toronto,  40 St George Street, Toronto, ON M5S 2E4, Canada}

{\sc Robert Haslhofer, Department of Mathematics, University of Toronto,  40 St George Street, Toronto, ON M5S 2E4, Canada}

\emph{E-mail:} wenkui.du@mail.utoronto.ca, roberth@math.toronto.edu

\end{document}